\documentclass[12pt]{amsart}

\usepackage{latexsym,amssymb,amsmath,amsthm,amscd,graphicx, xcolor, enumerate}

\usepackage{times}
\usepackage{microtype}
\usepackage{setspace}
\usepackage[margin=1.2in]{geometry}

\usepackage{cite}

\usepackage[colorlinks=true, pdfstartview=FitV, linkcolor=blue,
citecolor=blue, urlcolor=blue]{hyperref}

\numberwithin{equation}{section}
\allowdisplaybreaks

\newtheorem{theorem}{Theorem}[section]
\newtheorem{lemma}[theorem]{Lemma}

\theoremstyle{definition}

\newtheorem{remark}[theorem]{Remark}

\theoremstyle{remark}

\DeclareMathOperator*{\esssup}{ess \, sup}
\DeclareMathOperator*{\essinf}{ess \, inf}

\newcommand{\R}{{\mathbb R}}

\def\XXint#1#2#3{{\setbox0=\hbox{$#1{#2#3}{\int}$}
     \vcenter{\hbox{$#2#3$}}\kern-.5\wd0}}

\newcommand{\pp}{{p(\cdot)}}
\newcommand{\Pp}{\mathcal{P}}
\newcommand{\Lpp}{L^\pp}
\newcommand{\cpp}{{p'(\cdot)}}

\newcommand{\qq}{{q(\cdot)}}

\newcommand{\tB}{\tilde{B}}
\newcommand{\lp}{{\lambda(\cdot)}}
\begin{document}
\title[The Stein-Weiss inequality in  variable exponent Morrey spaces]
{The Stein-Weiss inequality \\ in  variable exponent Morrey spaces}

\author[D. Cruz-Uribe]{David Cruz-Uribe OFS}

\address{ David Cruz-Uribe, OFS\\
	Department of Mathematics\\
	University of Alabama \\
	Tuscaloosa, AL 35487 \\
	USA}

	\email{dcruzuribe@ua.edu}

\author[A. Ghorbanalizadeh]{Arash Ghorbanalizadeh}

\address{Arash Ghorbanalizadeh\\
	Department of Mathematics\\
 Nazarbayev University \\
	53 Kabanbay Batyr Ave, Astana 010000 \\
	Kazakhstan \\ 
    and Department of Mathematics\\
    Institute for Advanced Studies in Basic Sciences (IASBS) \\
    Zanjan, 45137-66731 \\
    Iran}

    \email{ arash.ghorbanalizadeh@nu.edu.kz; ghorbanalizadeh@iasbs.ac.ir}

\author[D. Suragan]{Durvudkhan Suragan}

\address{Durvudkhan Suragan \\
	Department of Mathematics \\
 Nazarbayev University \\
	53 Kabanbay Batyr Ave, Astana 010000 \\
	Kazakhstan}

\email{durvudkhan.suragan@nu.edu.kz}

\thanks{The first author is partially supported by a Simons Foundation Travel Support for Mathematicians Grant and by NSF grant DMS-2349550. The second two authors were funded by Nazarbayev University under Collaborative Research Program Grant 20122022CRP1601.}



\date{\today}

\begin{abstract}
In this paper we prove the Stein-Weiss inequality in variable exponent Morrey spaces over a bounded domain. Our work extends earlier results in the variable exponent Lebesgue and Morrey settings, and utilizes new proof techniques applicable to Morrey spaces. We build on the foundational paper \cite{AHS}, which  introduced Morrey spaces of variable exponents. As an application of our main result, we prove  Poincar\'e-type inequalities using the  approach of a recent paper~\cite{CrSu} by the first and third authors.
\end{abstract}

\maketitle{}

\section{Introduction}

In this paper we are interested in the boundedness  of fractional integral operators on various function spaces.  Given $0<\gamma<n$, define the fractional integral  $I_\gamma$ to be the convolution operator 
\[ I_\gamma f(x) = \int_{\R^n} \frac{f(y)}{|x-y|^{n-\gamma}}\,dy.  \]
The boundedness of this operator was  first considered on the Lebesgue spaces by Hardy and Littlewood~\cite{HardyLittlewood}, and  was later extended by other mathematicians, particularly Sobolev~\cite{Sobolev}, to higher dimensions.  The following result is often referred to  as the Hardy-Littlewood-Sobolev theorem:  for $0<\gamma<n$ and $1<p<\frac{n}{\gamma}$, define $q$ by $\frac{1}{p}-\frac{1}{q}=\frac{\gamma}{n}$.  Then 
\[  \|I_\gamma f\|_{L^q(\R^n)} \leq C\|f\|_{L^p(\R^n)}.  \]

Stein and Weiss~\cite{SteinWeiss} later proved a weighted version of this result.  

\begin{theorem} \label{thm:stein-weiss}
Fix $0<\gamma<n$, $1<p\leq q <\infty$, and constants $a$ and $b$ such that
\[ -\frac{n}{q} < a \leq b < \frac{n}{p'}. \]
Suppose further that 
\[ \frac{1}{p} - \frac{1}{q} = \frac{\gamma}{n} + \frac{a-b}{n}.  \]
Then there exists a constant $C>0$ such that for every function $f$,
\[ \| |\cdot|^a I_\gamma f\|_{L^q(\R^n)} \leq C\| |\cdot|^b f\|_{L^p(\R^n)}.
\]
\end{theorem}

\medskip

Bounds for the fractional integral operator on the classical Morrey spaces have been studied by a number of authors:  see~\cite{Adams, Olsen, Petre, Naka, NS22}. 
Recall that if  $1 \le p < \infty$ and $0\le \lambda \le n$, the Morrey space ${\mathcal M}_{p}^{\lambda}(\mathbb{R}^n)$ is the space of all $f\in L^p_{loc}(\R^n)$ such that 
$$
\|f\|_{{\mathcal M}_{p}^{\lambda}({\mathbb R}^n)}
:=
\sup _{x \in {\mathbb R}^n, ~ r>0} \left(  r^{-\lambda} \int_{B(x,r)} |f(y)|^{p} dy\right)^{\frac{1}{p}}< \infty, 
$$
where $B(x, r)= \{ y\in \R^n : |x-y|<r\}$.  We have the following result due to Spanne (see Peetre~\cite[Theorem~5.2]{Petre}).  
 
\begin{theorem}\label{spa}
Fix  $0 < \gamma < n$, $1 < p < \frac{n}{\gamma}$, and  define $q$ by $\frac{1}{p} - \frac{1}{q} =  \frac{\gamma}{n}$.  Suppose further that we have $0 < \lambda < \mu< n$ with
\[ \quad \frac{\lambda}{ p} = \frac{\mu}{ q}.
\]
Then there exists a constant $C>0$ such that for every function $f$,
\[
\| I_{\gamma} f \|_{M^\mu_q(\mathbb{R}^n)} \leq C \| f \|_{M^\lambda_p(\mathbb{R}^n)}.
\]
\end{theorem}

In the range $1<p<\frac{n-\lambda}{\gamma}$
Adams~\cite{Adams} proved a sharper result.

\begin{theorem}\label{Ada}
Fix  $0 < \gamma < n$, $0\leq \lambda \leq n$, $1 < p < \frac{n-\lambda}{\gamma}$, and  define $q$ by $\frac{1}{p} - \frac{1}{q} =  \frac{\gamma}{n-\lambda}$.  
Then there exists a constant $C>0$ such that for every function $f$,
\[
\| I_{\gamma} f \|_{ M_{q}^{\lambda}(\mathbb{R}^n)} \leq C \| f \|_{M_{p}^{\lambda}(\mathbb{R}^n)}.
\]
\end{theorem}

 Since \( \| f \|_{M^\mu_q(\mathbb{R}^n)} \lesssim \| f \|_{M_{q}^{\lambda}(\mathbb{R}^n)} \), Theorem \ref{Ada} improves Theorem \ref{spa} when \( 1 < p < \frac{n - \lambda}{\gamma} \).

A version of Theorem~\ref{thm:stein-weiss} in the Morrey spaces was recently proved by Kassymov, Ragusa,  Ruzhansky, and the third author~\cite{KRRS}.

\begin{theorem} \label{thm:KRRS}
    Fix  $0 < \gamma < n$, and $a,\,b$ such that 
$0 \leq b-a \leq \gamma$, and $p$ such that $1<p<\frac{n}{\gamma-(b-a)}$, and define $q$ by
$\frac{1}{p}-\frac{1}{q} = \frac{\gamma-(b-a)}{n-\lambda}$.  Suppose further that
\[ - \frac{n-\lambda}{q} < a \leq b < \frac{n}{p'}, \quad 0<\lambda < n- (\gamma-(b-a))p. \]
Then there exists a constant $C>0$ such that for every function $f$,
\[ \| |\cdot|^a I_\gamma f \|_{M^\lambda_q(\R^n)} \leq C\||\cdot|^b f  \|_{M^\lambda_p(\R^n)}. \]
\end{theorem}

\medskip

The boundedness of fractional integrals has also been considered on the variable Lebesgue spaces $L^\pp(\R^n)$.  These are a generalization of the classical $L^p$ spaces, replacing the constant exponent $p$ by an exponent function $\pp : \R^n \rightarrow [1,\infty]$.  For brevity, we defer precise definitions to Section~\ref{section:prelim} below.   It was shown in~\cite{MR2414490} that if $1<p_-\leq p_+<n/\gamma$, and if $\pp$ is log-H\"older continuous, then for $\qq$ defined by 
\[ \frac{1}{p(x)} - \frac{1}{q(x)} = \frac{\gamma}{n}, \]

the Hardy-Sobolev inequality holds with variable exponents:
\[ \|I_\gamma f\|_{L^\qq(\R^n)} \leq C\|f\|_{L^\pp(\R^n)}. \]
For earlier results see~\cite{MR2414490,Cruz} and the references they
contain.

Samko~\cite{Samko} proved a version of the Stein-Weiss inequality in
the variable Lebesgue spaces on a bounded domain.  Here we state a somewhat simpler version of his result.

\begin{theorem}\label{thm:Samko}
Given $0<\gamma<n$, and a bounded open set  \( \Omega \subset \mathbb{R}^n \), fix \( x_0 \in \overline{\Omega} \). Suppose that \( \pp \in \Pp(\Omega) \) is such that $1<p_-\leq p_+< \frac{n}{\gamma}$ and $\pp \in LH_0(\Omega)$. 
Define $\qq$ by $\frac{1}{q(x)} = \frac{1}{p(x)} - \frac{\gamma}{n}$.  Fix constants $\mu$ and $\nu$ such that
\[ \gamma p(x_0) - n < \nu < n \big(p(x_0) - 1\big), \quad \mu = \frac{q(x_0)}{p(x_0)} \nu. \]
Then there exists a constant $C>0$ such that for every function $f$,
\[
\| |\cdot-x_0|^{\frac{\mu}{\qq}}I_{\gamma} f \|_{L^{q(\cdot)}(\Omega)} 
\le C \||x - x_0|^{\frac{\nu}{\pp}} f \|_{L^{\pp}(\Omega)}.
\]

\end{theorem}

More recently, the first and third authors~\cite[Theorem~3.1]{CrSu}, building on the work of Melchiori and Pradolini~\cite{MP18}, proved a version of the Stein-Weiss inequality in variable Lebesgue spaces.

\begin{theorem} \label{thm:var-stein-weiss}
Fix $0<\gamma<n$.  Given $\pp,\,\qq \in \Pp(\R^n)$, suppose that
$\pp,\,\qq \in LH(\R^n)$, $1<p_-\leq p_+<\infty$, $1<q_-\leq
q_+<\infty$, and $p(x)\leq q(x)$ for all $x\in \R^n$.  Given constants
$a,\, b$ that satisfy
\begin{equation*} \label{eqn:vsw1}
 -\frac{n}{q_+} < a\leq b < \frac{n}{(p_-)'}, 
\end{equation*}
suppose further that for all $x\in \R^n$, 
\begin{equation*} \label{eqn:vsw2}
  \frac{1}{p(x)}-\frac{1}{q(x)} = \frac{\gamma}{n} +
  \frac{a-b}{n}.
\end{equation*}
Then for all $f$ there exists a constant $C>0$ such that 
\[ \||\cdot|^a I_\gamma f\|_{\qq} \leq C \||\cdot|^b f\|_{\pp}. \]
\end{theorem}

\medskip

In this paper we generalize Theorems~\ref{thm:KRRS},~\ref{thm:Samko},
and~\ref{thm:var-stein-weiss} by proving a version of the Stein-Weiss
inequality in variable exponent Morrey spaces defined on a bounded set
$\Omega$.  The variable exponent Morrey spaces were introduced by Almeida,
Hasanov, and Samko~\cite{AHS}; they proved bounds for the
Hardy-Littlewood maximal operator, fractional maximal operators, and fractional integrals on bounded
subsets of $\R^n$.   Additional
results on bounded domains can be found in~\cite{ GuHaSa, GuHaSa-1,
  GuHaSa-2, Has, MT, SaTe, SaTe-1}.  Our main result is the following.

\begin{theorem}\label{thm:weight}
Fix  \( 0 < \gamma < n \).    Let \(\Omega\subset \R^n\) be a bounded
domain and fix 
  \(x_0 \in \Omega\).  Given
  \( \pp, q(\cdot) \in \mathcal{P}(\Omega) \), suppose  that
  \( \pp, q(\cdot) \in LH_0(\Omega) \),
  \( 1 < p_{-} \le p_{+} < \infty \),
  \( 1 < q_{-} \le q_{+} < \infty \), and \( p(x) \le q(x) \) for all
  \( x \in \Omega \). Let $\lp : \Omega \to [0,n]$ be such that $0 < \lambda_- \leq \lambda_+ <n$ and  $\lp \in LH_0(\Omega)$.  Let $a,\,b$ be constants such
  that
  \begin{gather}
        0 \leq  b -a \le \gamma, \label{eqn:condA} \\
        \frac{\lambda_{+}-n}{q_+} < a \le b < \frac{n}{ (p')_{+}}, \label{eqn:condB} \\
        0 < \lambda_{-} \leq \lambda(x) < n - (\gamma - b + a)p(x), \label{eqn:condC} 
  \end{gather} 
    and
\begin{equation}\label{eqn:condD}
 \frac{\gamma + a- b}{n - \lambda(x)}= \frac{1}{p(x)} - \frac{1}{q(x)}.
\end{equation}
Then, there exists a constant $C>0$ such that for every function $f$,
\begin{equation}\label{SWI}
\| |x-x_0|^{a} I_{\gamma} f \|_{M_{q(\cdot)}^{\lp}(\Omega)} \leq C \| |x-x_0|^{b} f \|_{M_{\pp}^{\lp}(\Omega)}.
\end{equation}
\end{theorem}

\begin{remark}
  Our proof, while based on that of Theorem~\ref{thm:KRRS}, does not
  seem to extend to unbounded domains.  It is an open problem to prove
  a Stein-Weiss type inequality on variable exponent Morrey spaces defined on
  all of $\R^n$.  For unweighted bounds for fractional integrals and fractional
  maximal operators in this setting, see Ho~\cite{Ho}.
\end{remark}


The remainder of the paper is organized as follows.  In
Section~\ref{section:prelim} we gather some basic results about
variable Lebesgue spaces, variable exponent Morrey spaces, fractional integral
operators and the associated fractional maximal operators.  In
Section~\ref{section:main-proof} we prove Theorem~\ref{thm:weight}.
Finally, in Section~\ref{section:applications} we give an application
of our main result and use it to prove Sobolev and
Poincar\'e inequalities in the variable exponent Morrey spaces.

\section{Preliminaries}
\label{section:prelim}

In this section we establish some basic notation and preliminary
results.  Throughout, $n$ will denote the dimension of the underlying
space $\R^n$, and $\Omega \subseteq \R^n$ will be a bounded, open
set.  For $x\in \Omega$ and $r>0$, we define $B(x,r)=\{ y \in \R^n :
|x-y|<r\}$ and $\tB(x,r) = B(x,r) \cap \Omega$.

By $C$ we will mean a constant whose values depend on the
underlying parameters and whose value may change from line to line.
If we write $A\lesssim B$, we mean that there exists a constant $c$
such that $A\leq cB$.  

\subsection*{Variable Lebesgue spaces}
Here we define the variable Lebesgue spaces and give some basic properties.
For more detailed information and proofs of these properties, we refer the reader to~\cite{Cruz}.  

A variable exponent is a measurable function $\pp : \Omega \rightarrow
[1,\infty]$.  We denote the set of exponent functions by
$\mathcal{P}(\Omega)$.  Let
\[ p^{+}= \esssup_{x \in \Omega} p(x) \quad \text{and} \quad p^{-}= \essinf_{x \in \Omega} p(x). \]

A function $r(\cdot) : \Omega \to \R$ is said to be locally
Log-H\"older continuous, denoted by $r(\cdot) \in
LH_0(\Omega)$, if there exists a constant $C_0$ such that
\[ |r(x) - r(y)| \leq \frac{C_0}{-\log(|x - y|)}
\]
for all \( x, y \in \Omega \) with \( |x - y| < \frac{1}{2} \).

Given an exponent $\pp \in \Pp(\Omega)$, define the associated modular by
\[ \rho_\pp(f) = \int_{\Omega\setminus \Omega_\infty}
  |f(x)|^{p(x)}\,dx
  + \|f\|_{L^\infty(\Omega_\infty)}, 
\]
where $\Omega_\infty=\{ x\in \Omega : p(x)=\infty\}$. 

\begin{remark}
  When $p_+<\infty$, the set $\Omega_\infty$ is empty and the modular
  reduces to the first integral.
\end{remark}

Define $\Lpp(\Omega)$ to be the collection of measurable functions
with domain $\Omega$  such that
\[ \|f\|_\pp = \|f\|_{\Lpp(\Omega)}
  = \inf\{ \lambda >0 : \rho(f/\lambda) \leq 1 \}.  \]
Then $\|\cdot\|_\pp$ is a norm and $\Lpp(\Omega)$ is a Banach function
space.


Given $\pp \in \Pp(\Omega)$,  define the conjugate exponent \(
p'(\cdot) \)pointwise by 
\[
\frac{1}{p(x)}+\frac{1}{p'(x)} = 1, 
\]  
with the conventions that $\frac{1}{\infty}=0$. 
H\"older's inequality extends to the variable Lebesgue spaces:
\begin{equation}\label{Holder}
\int_\Omega |f(x) g(x)| \, dx \leq  
4 \|f\|_{\pp} \|g\|_{p'(\cdot)}.   
\end{equation}
The following generalized H\"older's inequality~\cite[Corollary 2.28.]{Cruz} also holds:  if \( \pp, q(\cdot), r(\cdot) \in \mathcal{P} (\Omega) \) satisfy  
\[
\frac{1}{r(x)} = \frac{1}{p(x)} + \frac{1}{q(x)}, \quad \text{for all } x \in \Omega,
\]  
then there exists a constant \( K=K(\pp,\qq)\geq 1 \) such that for any \( f \in L^{\pp}(\Omega) \) and \( g \in L^{q(\cdot)}(\Omega) \), 
\begin{equation}\label{Gholder}
\|fg\|_{L^{r(\cdot)}(\Omega)} \leq K \|f\|_{L^{\pp}(\Omega)} \|g\|_{L^{q(\cdot)}(\Omega)}.
\end{equation}

The following lemmas will be used repeatedly in our proofs.   Most of them rely heavily on the fact that $\Omega$ is bounded.  The first is
from~\cite[Lemma~4]{AHS}.  

\begin{lemma}\label{Lemma 4}
  Given a bounded open set $\Omega\subset
  \R^n$, suppose $\lp\in LH_0(\Omega)$. 
Then there exists a constant $C=C(\lambda,\Omega)$ such that for any $s>0$ and all $x, y \in \Omega$ such that $|x - y| \leq s$,
\[
C^{-1} s^{-\lambda(y)} \leq s^{-\lambda(x)} \leq C s^{-\lambda(y)}.
\]
\end{lemma}

The second is from~\cite[Lemma~6]{AHS}.

\begin{lemma}\label{Lemma 6}
  Given a bounded open set $\Omega\subset
  \R^n$, suppose $\pp \in \Pp(\Omega)$ satisfies $\pp\in LH_0(\Omega)$. 
Then there exists a constant $C=C(\pp,\Omega)$ such that for all $x\in
\Omega$ and $r>0$,
\[ 
\|\chi_{B(x, r)}\|_{{\pp}} \leq C |B(x, r)|^{\frac{1}{p(x)}} .
\]
\end{lemma}

The third is a consequence of well-known properties of the variable exponent norm.  

\begin{lemma} \label{lemma:char-estimate}
Given a bounded open set $\Omega \subset \R^n$, and $s(\cdot) \in \Pp(\Omega)$, 
$\|\chi_\Omega\|_{L^{s(\cdot)}(\Omega)} \leq 1+|\Omega|$.
\end{lemma}

\begin{proof}
    If $\|\chi_\Omega\|_{L^{s(\cdot)}(\Omega)} \leq 1$, then we are done, so assume the opposite inequality holds.  Then by \cite[Corollary~2.22]{Cruz},
    \[ \|\chi_\Omega\|_{L^{s(\cdot)}(\Omega)} \leq \rho_{s(\cdot)}(\chi_\Omega)
    = \int_{\Omega\setminus\Omega_\infty} \,dx + \|\chi_\Omega\|_{L^\infty(\Omega_\infty)} \leq |\Omega|+1. \]
\end{proof}

Finally, we will need the following integral estimate, which is a consequence of Lemma~\ref{Lemma 4}.

\begin{lemma} \label{lemma:integral-bound}
Given a bounded open set $\Omega\subset
  \R^n$, suppose $\lp :\Omega \to [0,n]$ is such that $0<\lambda_-\leq \lambda_+<n$ and  $\lp\in LH_0(\Omega)$.  Then for
  every  $x_0 \in \Omega$ and $r>0$.  
  \begin{equation}\label{Bound}
\int_{\tB(x_0,r)} |x-x_0|^{\lambda(x) - n} dx \le  C r^{\lambda(x_0)}.
\end{equation}
\end{lemma}

\begin{proof}
Since $\lp\in LH_0(\Omega)$, $n-\lp \in LH_0(\Omega)$.  Therefore, by Lemma~\ref{Lemma 4}, taking $s=|x-x_0|$ in the lemma, and integrating by parts,  we have that
\[ 
   \int_{\tB(x_0,r)} |x-x_0|^{\lambda(x) - n} dx 
 \le C(\lambda, \Omega) \int_{B(x_0,r)} |x-x_0|^{\lambda(x_0) - n} dx 
 =  C(n,\lambda , \Omega) r^{\lambda(x_0)} .
 \]
%
\end{proof}

\subsection*{Variable exponent Morrey spaces}

Our definition of the variable exponent Morrey spaces is taken
from~\cite{AHS} and we refer the reader there for more information.  
Given $\lp : \Omega \to [0,n]$  measurable and $\pp \in
\Pp(\Omega)$, define the
modular 
\[
I_{\pp, \lp}(f) := \sup_{x \in \Omega, \, r > 0}
r^{-\lambda(x)} \int_{\tilde{B}(x, r)} |f(y)|^{p(y)} \, dy.  
\]
Then the variable exponent Morrey space \(
M_{\pp}^{\lp}(\Omega) \) is defined as the set of
all measurable functions $f$ on $\Omega$ such that 
\[
\|f\|_{M_{\pp}^{\lp}(\Omega)} = \inf \left\{ \eta > 0
  : I_{\pp, \lp}\left( \frac{f}{\eta} \right) \le 1
\right\} < \infty.
\]

To prove norm inequalities in the variable exponent Morrey spaces, we need the
following lemma.  This property is well known in the variable Lebesgue
spaces (cf.~\cite[Section~3.4]{Cruz}).

\begin{lemma} \label{lemma:modular-to-norm}
Given a function $g$, $\|g\|_{M^{\lp}_\pp(\Omega)} \leq C_1$ if and only if 
 $I_{\pp,\lp}(g) \leq C_2$ for some constants $C_1,\,C_2>0$.  
 \end{lemma}

  \begin{proof}
First, suppose that  $I_{\pp,\lp}(g) \leq C_2$. We may assume $C_2\geq 1$, (for \(0 < C_{2} < 1\), one can choose \(C_{1} = 1\)). Then we have that
    \begin{equation*}
      1
       \geq
       C_2^{-1} \sup_{x \in \Omega, \, r > 0}
        r^{-\lambda(x)} \int_{\tilde{B}(x, r)} |g(y)|^{p(y)} \, dy 
       \geq \sup_{x \in \Omega, \, r > 0}
        r^{-\lambda(x)} \int_{\tilde{B}(x, r)}
        |C_2^{-\frac{1}{p_-}}g(y)|^{p(y)} \, dy.
    \end{equation*}
    Thus, by the definition of the variable exponent Morrey space norm,
    $\|C_2^{-\frac{1}{p_-}} g\|_{M^\lp_\pp(\Omega)} \leq 1$, which
    gives us the desired estimate with $C_1=C_2^{\frac{1}{p_-}}$. 

    Now suppose that $\|g\|_{M^{\lp}_\pp(\Omega)} \leq C_1$. Without loss of generality, we may assume $C_1\geq 1$.  Then, again by the definition, we have that 
    \[ 1 \geq \sup_{x \in \Omega, \, r > 0}
        r^{-\lambda(x)} \int_{\tilde{B}(x, r)} |C_1^{-1}g(y)|^{p(y)} \, dy
        \geq C_1^{- p_+} \sup_{x \in \Omega, \, r > 0}
        r^{-\lambda(x)} \int_{\tilde{B}(x, r)} |g(y)|^{p(y)} \, dy.\]
    This gives us the desired inequality with $C_2=C_1^{p_+}$.
  \end{proof}

  We will need to show that with our hypotheses $\frac{\lp}{\pp}$ is log-H\"older continuous.

  \begin{lemma} \label{lemma:l/p-LH0}
  Given $\pp\in \Pp(\Omega)$ such that $p_+<\infty$ and $\pp \in LH_0(\Omega)$, suppose that $\lp : \Omega \to [0,n]$ satisfies $\lp \in LH_0(\Omega)$.  Then $\frac{\lp}{\pp}\in LH_0(\Omega)$. 
  \end{lemma}

  \begin{proof}
      Fix $x,\,y \in \Omega$ with $|x-y|<\frac{1}{2}$.  Then 
      \[  \left|\frac{\lambda(x)}{p(x)}-\frac{\lambda(y)}{p(y)}\right|
      \leq 
      \left|\frac{\lambda(x)}{p(x)}-\frac{\lambda(y)}{p(x)}\right| + \left|\frac{\lambda(y)}{p(x)}-\frac{\lambda(y)}{p(y)}\right|
      \leq 
      \frac{1}{p_-}|\lambda(x)-\lambda(y)| + \frac{\lambda_+}{p_-^2}|p(x)-p(y)|.  \]
      It follows immediately that $\frac{\lp}{\pp}\in LH_0(\Omega)$.
  \end{proof}
We will also need an equivalent form for the norm, which was essentially
shown in~\cite{AHS}.

\begin{lemma} \label{lemma:equiv-norm}
If $\lp, \, \pp \in LH_0(\Omega)$, and $\lambda_{+}<\infty$, $p_+<\infty$, then 
\[  \|f\|_{M_{\pp}^{\lp}(\Omega)}
  \approx
  \sup_{x \in \Omega, \, r > 0} r^{-\frac{\lambda(x)}{p(x)}} \left\|
     f \, \chi_{\tilde{B}(x, r)}
  \right\|_{\pp}. \]
\end{lemma}

\begin{proof}
By \cite[Lemma 5]{AHS}, we have that
\[
\|f\|_{M_{\pp}^{\lp}(\Omega)} \approx 
\sup_{x \in \Omega, \, r > 0}  \left\| r^{-\frac{\lp}{\pp}}  f \, \chi_{\tilde{B}(x, r)} \right\|_{\pp}.
\]
By Lemma~\ref{lemma:l/p-LH0}, $\frac{\lp}{\pp}\in LH_0(\Omega)$, and so by
Lemma~\ref{Lemma 4}, we have that for all $y\in \tB(x,r)$,
\[  r^{-\frac{\lambda(y)}{p(y)}} \approx
  r^{-\frac{\lambda(x)}{p(x)}}. \]
The desired equivalence follows at once.
\end{proof}

Finally, we need bounds for  maximal operators on variable exponent Morrey spaces.  
For $0\leq \sigma <n$, define 
\[
M_\sigma f(x) = \sup_{\substack{z\in \Omega \\r > 0}}  \frac{1}{|B(z, r)|^{1-\frac{\sigma}{n}}} \int_{B(z, r)} |f(y)| \, dy \cdot \chi_{B(z,r)}(x).
\]
When $\sigma=0$ this is the Hardy-Littlewood maximal operator, and when $\sigma>0$ the fractional maximal operator.
The following result was proved in \cite[Theorems~2, 3 \& Corollary~2]{AHS}.  

\begin{theorem}\label{Theorem 2}
Given a bounded, open set $\Omega\subset \mathbb{R}^n$, let $\pp \in \Pp(\Omega)$ be
such that $1 < p_{-} \leq p_{+} < \infty$ and $\pp \in
LH_0(\Omega)$.   Let
$\lp$ be a measurable function on $\Omega$ such that
$0 \leq \lambda_- \leq \lambda_{+} < n$ and
\begin{equation*}
 \sup_{x \in \Omega} [\lambda(x) + \sigma p(x)] < n.
\end{equation*}
Fix $\sigma$, $0\leq \sigma < n$, and define $\qq \in \Pp(\Omega)$ by
\[
\frac{1}{q(x)} = \frac{1}{p(x)} - \frac{\sigma}{n - \lambda(x)}.
\]
Then the maximal operator $M_\sigma$ is bounded from $M^\lp_\pp(\Omega)$ to $M^\lp_\qq(\Omega)$. The same is also true for the fractional integral operator $I_\sigma$.
\end{theorem}



\section{Proof of Theorem~\ref{thm:weight}}
\label{section:main-proof}

To prove this result, we make some reductions.  Fix $x_0 \in
\Omega$.  First, since the
$M_{\pp}^{\lp}(\Omega)$ norm is homogeneous, it will suffice to prove
this result for $f$ such that $\||\cdot-x_0|^b f\|_{M_{\pp}^{\lp}(\Omega)}=1$, and
so we need to show that for some $C\geq 1$, $\||\cdot-x_0|^a I_\gamma
f\|_{M_{\qq}^{\lp}(\Omega)}\leq C$.  (If $C<1$, there is nothing to prove.)
Second, since $I_\gamma$ is a positive operator, we may also assume that $f$
is nonnegative.  

Third, for each $x\in \Omega$ we  decompose
$I_\gamma f(x)$ as follows:
\begin{align*}
  I_\gamma f(x)
 &= \int_{\Omega} \frac{f(y)}{|x - y|^{n - \gamma}} \, dy \\
&= \int_{\tB(x_0,\frac{|x-x_0|}{2})} \frac{f(y)}{|x - y|^{n - \gamma}} \, dy \\
& \qquad + \int_{\tB(x_0, 2|x-x_0|)\setminus \tB(x_0,\frac{|x-x_0|}{2})} \frac{f(y)}{|x - y|^{n - \gamma}} \, dy 
\\
&  \qquad + \int_{\Omega \setminus \tB(x_0,2|x-x_0|)} \frac{f(y)}{|x - y|^{n - \gamma}} \, dy \\
&= J_1(x) + J_2(x) + J_3(x).
\end{align*}
Thus, it will suffice to show that for $1\leq i \leq 3$,
$\||\cdot-x_0|^a J_i\|_{M_{\qq}^{\lp}(\Omega)}\leq C$.

Finally, by Lemma~\ref{lemma:modular-to-norm} it will suffice to prove
that there exists $C\geq 1$ such that for every $z\in \Omega$ and
every $r>0$,
\begin{equation} \label{eqn:Ji-estimate}
 r^{-\lambda(z)}\int_{\tB(z,r)} \big[|x-x_0|^a
  J_i(x)\big]^{q(x)}\,dx \leq C. 
\end{equation}

\medskip

We estimate each of these in turn.  We first consider the $J_1$ term.
In $\tB(x_0, \frac{|x-x_0|}{2})$ we have that  \( |y-x_0| <
\frac{|x-x_0|}{2} \), which in turn implies that \( |x - y| \geq
|x-x_0| - |y-x_0| > \frac{|x-x_0|}{2} \).  Hence,
\begin{align*}
  J_1(x)
  &\leq \int_{\tB(x_0,\frac{|x-x_0|}{2})}
    \frac{f(y)}{(\frac{|x-x_0|}{2})^{n - \gamma}} \, dy  \\
  &= \frac{2^{n-\gamma}}{|x-x_0|^{n - \gamma}}
    \int_{\tB(x_0,\frac{|x-x_0|}{2})} f(y) \, dy \\
  & \lesssim {\frac{1}{|x-x_0|^{n - \gamma}} }\sum_{k=1}^{\infty}
    \int_{\tB(x_0,2^{-k}|x-x_0|)\setminus
     \tB(x_0, 2^{-k-1}|x-x_0|)} f(y) \, dy \\
  & \lesssim {\frac{1}{|x-x_0|^{n - \gamma}}} \sum_{k=1}^{\infty}
    (2^{-k}|x-x_0|)^{-b} \int_{\tB(x_0,2^{-k}|x-x_0|)} |y-x_0|^{b} {f(y)}  \, dy.
\intertext{By  H\"older’s inequality \eqref{Holder} and Lemma \ref{Lemma 6}, } %
  &\lesssim {\frac{1}{|x-x_0|^{n - \gamma}}}\sum_{k=1}^{\infty} (2^{-k}|x-x_0|)^{-b}
   \| \chi_{\tB(x_0,2^{-k}|x-x_0|)} |\cdot-x_0|^{b } f
    \|_{L^\pp(\Omega)}
   \| \chi_{B(x_0,2^{-k}|x-x_0|)} \|_{L^{p'(\cdot)}(\Omega)}
  \\
 &\lesssim {\frac{1}{|x-x_0|^{n - \gamma}}}  \sum_{k=1}^{\infty}
   (2^{-k}|x-x_0|)^{-b+\frac{n}{p'(x_0)}}
    \| \chi_{\tB(x_0,2^{-k}|x-x_0|)} |\cdot-x_0|^{b } f
    \|_{L^\pp(\Omega)}.
  \\
   \intertext{If we now apply the equivalent norm from
  Lemma~\ref{lemma:equiv-norm} and our norm assumption on $f$ we get}
  &= {\frac{1}{|x-x_0|^{n - \gamma}}} \sum_{k=1}^{\infty}
    (2^{-k}|x-x_0|)^{-b+\frac{n}{p'(x_0)}+\frac{\lambda(x_0)}{p(x_0)}}  
 \\
 & \qquad \qquad  \times
   (2^{-k}|x-x_0|)^{\frac{-\lambda(x_0)}{p(x_0)}} 
   \| \chi_{\tB(x_0,2^{-k}|x-x_0|)} |\cdot-x_0|^{b } f
    \|_{L^\pp(\Omega)}
  \\
  &\lesssim {\frac{1}{|x-x_0|^{n - \gamma}}\times} \sum_{k=1}^{\infty}
    (2^{-k}|x-x_0|)^{-b+\frac{n}{p'(x_0)}+\frac{\lambda(x_0)}{p(x_0)}}
    \||\cdot-x_0|^{b }f \|_{M_{\pp}^{\lp}(\Omega)} \\
  & = {\frac{1}{|x-x_0|^{n - \gamma}}} \sum_{k=1}^{\infty}
    (2^{-k}|x-x_0|)^{-b+\frac{n}{p'(x_0)}+\frac{\lambda(x_0)}{p(x_0)}}; \\
    \intertext{by the second inequality in \eqref{eqn:condB}, $-b+\frac{n}{p'(x_0)}\geq -b+\frac{n}{(p')_+}>0$, so the series converges and we get}
  & \lesssim {\frac{1}{|x-x_0|^{n - \gamma}}}  |x-x_0|^{-b+\frac{n}{p'(x_0)}+\frac{\lambda(x_0)}{p(x_0)}} \\
  & \lesssim  {|x-x_0|^{-b + \gamma -\frac{n-\lambda(x)}{p(x)}}}.
\end{align*}
The last inequality follows from Lemma~\ref{Lemma 4}, with $s=|x-x_0|$
and Lemma~\ref{lemma:l/p-LH0}. 
 
We can now estimate~\eqref{eqn:Ji-estimate} for $J_1$.  Fix $z\in
\Omega$ and $r>0$.  Then by the above estimate, the identity~\eqref{eqn:condD}, and Lemma~\ref{lemma:integral-bound}, we get
\begin{align*}
  r^{-\lambda(z)}\int_{\tB(z,r)} \big[|x-x_0|^a
  J_1(x)\big]^{q(x)}\,dx
  & \lesssim
     r^{-\lambda(z)}\int_{\tB(z,r)} \big[|x-x_0|^a
    |x-x_0|^{-b + {\gamma} -\frac{n-\lambda(x)}{p(x)}}\big]^{q(x)}\,dx \\
  & \leq   r^{-\lambda(z)}\int_{B(z,r)} |x-x_0|^{\lambda(x)-n}\,dx \\
  & \lesssim 1.
\end{align*}
Since the implicit constant is independent of $z$ and $r$, we have
proved~\eqref{eqn:Ji-estimate} for $J_1$.

\medskip

Next we will consider~\eqref{eqn:Ji-estimate} for $J_3$,
\[
J_3(x) = \int_{\Omega \setminus\tB(x_0,2|x-x_0|)} \frac{f(y)}{|y - x|^{n-\gamma}} \, dy.
\]
For all $y$ such that \( 2|x-x_0| \leq |y-x_0| \) we have that
\[
|y-x_0| = |y -x_0- x + x| \leq |y - x| + |x-x_0| \leq |y - x| + \frac{|y-x_0|}{2};
\]
hence,
\begin{equation} \label{eqn:J3-est}
\frac{|y-x_0|}{2} \leq |y - x|.
\end{equation}

Moreover, by the first inequality in~\eqref{eqn:condB} we have that
\[ \frac{n-\lambda(x)}{q(x)} \geq \frac{n-\lambda(x)}{q_+}  \geq \frac{n-\lambda_+}{q_+} > - a.\]
Hence,
\[
\frac{1}{q(x)} = \frac{1}{p(x)} - \frac{\gamma - b + a}{n - \lambda(x)} 
< \frac{1}{p(x)} + \frac{1}{q(x)} - \frac{\gamma - b}{n - \lambda(x)},
\]
which in turn implies  
\begin{equation}\label{lazim}
\frac{n - \lambda(x)}{p(x)} - \gamma + b > 0.  
\end{equation}

We can now argue much as we did in our estimate for $J_1$: by inequality~\eqref{eqn:J3-est}
\begin{align*}
    J_3(x)
    & \lesssim \int_{\Omega \setminus \tB(x_0,2|x-x_0|)} \frac{f(y)}{|y - x_0|^{n-\gamma}} \, dy \\
    & = \sum_{k=1}^\infty \int_{\tB(x_0,2^{k+1}|x-x_0|)\setminus \tB(x_0,2^k|x-x_0|)}
        \frac{|y-x_0|^bf(y)}{|y - x_0|^{n-\gamma+b}} \, dy \\
    & \leq \sum_{k=1}^\infty (2^k|x-x_0|)^{\gamma-n-b}
    \int_{\tB(x_0,2^{k+1}|x-x_0|)}
    |y-x_0|^b {f(y)} \, dy. \\
    \intertext{By H\"older's inequality \eqref{Holder} and Lemma~\ref{Lemma 6}, we get }
    & \leq \sum_{k=1}^\infty (2^k|x-x_0|)^{\gamma-n-b}
    \|\chi_{\tB(x_0,2^{k+1}|x-x_0|)}\|_{L^\cpp(\Omega)} \||\cdot-x_0|^b f\chi_{\tB(x_0,2^{k+1}|x-x_0|)}\|_{L^\pp(\Omega)} \\
     & \lesssim \sum_{k=1}^\infty (2^k|x-x_0|)^{\gamma-n-b+\frac{n}{p'(x_0)}+\frac{\lambda(x_0)}{p(x_0)}} \\
     & \qquad \qquad \times (2^{k+1}|x-x_0|)^{-\frac{\lambda(x_0)}{p(x_0)}}
     \||\cdot-x_0|^b f\chi_{\tB(x_0,2^{k+1}|x-x_0|)}\|_{L^\pp(\Omega)}\\
     & =  \sum_{k=1}^\infty (2^k|x-x_0|)^{\gamma-b-\frac{n-\lambda(x_0)}{p(x_0)}} (2^{k+1}|x-x_0|)^{-\frac{\lambda(x_0)}{p(x_0)}}
     \||\cdot-x_0|^b f\chi_{\tB(x_0,2^{k+1}|x-x_0|)}\|_{L^\pp(\Omega)}. \\
     \intertext{By Lemma~\ref{lemma:l/p-LH0}, $\frac{\lp}{\pp}\in LH_0(\Omega)$; hence, by Lemmas~\ref{Lemma 4} and~\ref{lemma:equiv-norm},}
     & \lesssim \sum_{k=1}^\infty (2^k|x-x_0|)^{\gamma-b-\frac{n-\lambda(x)}{p(x)}}
     \||\cdot-x_0|^b f\|_{M^\lp_\pp(\Omega)}.
     \end{align*}
     By inequality~\eqref{lazim}, the series converges, and so by our assumption on the norm of $f$ we have shown that
     \[ J_3(x)  \lesssim (|x-x_0|)^{\gamma-b-\frac{n-\lambda(x)}{p(x)}}. \]

We can now estimate \eqref{eqn:Ji-estimate} for $J_3$.  Fix $z\in \Omega$ and $r>0$; then, arguing exactly as we did for $J_1$, we have that
\begin{multline*}
    r^{-\lambda(z)}\int_{\tB(z,r)} \big[|x-x_0|^a
  J_3(x)\big]^{q(x)}\,dx \\ 
    \lesssim
     r^{-\lambda(z)}\int_{\tB(z,r)} \big[|x-x_0|^a
    |x-x_0|^{-b + n -\frac{n-\lambda(x)}{p(x)}}\big]^{q(x)}\,dx 
   \lesssim 1. 
\end{multline*} 

\medskip

Finally, we consider~\eqref{eqn:Ji-estimate} for $J_2$,
\[
J_2(x) = \int_{\tB(x_0, 2|x-x_0|)\setminus \tB(x_0,\frac{|x-x_0|}{2})} \frac{f(y)}{|x - y|^{n - \gamma}} \, dy.
\]
Temporarily, for brevity let $A=\tB(x_0, 2|x-x_0|)\setminus \tB(x_0,\frac{|x-x_0|}{2})$.  We will consider two cases.  Suppose first that $b>a$.  If $y\in A$, then 
\[ |x-y| \leq |x-x_0| + |x_0-y| \leq 3|x-x_0|.  \]
Therefore, we can cover $A$ by dyadic annuli.  More precisely,
\[ A = \bigcup_{k=0}^\infty  D_k, \]
where $D_k = \{ y \in A : 3\cdot2^{-k-1} |x-x_0| \leq |x-y| < 3 \cdot 2^{-k}|x-x_0|\}$.   
By \eqref{eqn:condA}, $\gamma \geq b-a>0$, so we can define $\sigma\geq 0$ by $\gamma=\sigma+b-a$.  We can now estimate as follows:
\begin{align*}
    J_2(x)
    & \le \sum_{k=0}^\infty \int_{D_k} \frac{f(y)}{|x - y|^{n - \gamma}} \, dy \\
    & \lesssim \sum_{k=0}^\infty \int_{D_k} (2^{-k}|x-x_0|)^{\gamma-n} f(y)\, dy \\
    & = \sum_{k=0}^\infty 2^{-k(b-a)} |x-x_0|^{-a} (2^{-k}|x-x_0|)^{\sigma-n} \int_{D_k} |x-x_0|^{b} f(y)\, dy \\
    & \lesssim \sum_{k=0}^\infty 2^{-k(b-a)} |x-x_0|^{-a} (2^{-k}|x-x_0|)^{\sigma-n} \int_{D_k} |y-x_0|^{b} f(y)\, dy. \\
    \intertext{Since $D_k \subset \tB(x,3\cdot 2^{-k}|x-x_0|)$, by the definition of the fractional maximal operator (or the maximal operator if $\sigma=0$), }
    & \lesssim \sum_{k=0}^\infty 2^{-k(b-a)} |x-x_0|^{-a} M_\sigma (|\cdot-x_0|^b f)(x) \\
    & \lesssim |x-x_0|^{-a} M_\sigma (|\cdot-x_0|^b f)(x).
\end{align*}
The last inequality holds since $b-a>0$ and so the series converges.

We can now estimate~\eqref{eqn:Ji-estimate} for $J_2$: for any $z\in \Omega$ and $r>0$,
\begin{multline*}
    r^{-\lambda(z)} \int_{\tB(z,r)} [|x-x_0|^a J_2(x)]^{q(x)} \,dx \\
  \lesssim r^{-\lambda(z)} \int_{\tB(z,r)} M_\sigma(|\cdot-x_0|^b f)(x)^{q(x)} \,dx
  \leq I_{\pp,\lp}(M_\sigma(|\cdot-x_0|^b f).  
\end{multline*} 
By our normalization assumption on $f$, $\||\cdot-x_0|^b f\|_{M^\lp_\pp(\Omega)}=1$, and so by Theorem~\ref{Theorem 2}, 
$\|M_\sigma(|\cdot-x_0|^b f)\|_{M^\lp_\qq(\Omega)}\leq C$. Therefore, by Lemma~\ref{lemma:modular-to-norm}, $I_{\pp,\lp}(M_\sigma(|\cdot-x_0|^b f)\leq C$.  This completes the estimate for $J_2$ when $b>a$.  

Now suppose that $b=a$.  Then for every $y\in A$, we have that 
\[ 1 = |x-x_0|^{b-a} =|x-x_0|^{-a} |x-x_0|^b \approx |x-x_0|^{-a}|y-x_0|^b. \]
Therefore, 
\begin{multline*} 
J_2(x) 
\lesssim |x-x_0|^{-a} \int_{A} \frac{|y-x_0|^bf(y)}{|x - y|^{n - \gamma}} \, dy
\\ \leq |x-x_0|^{-a} \int_{A} \frac{|y-x_0|^bf(y)}{|x - y|^{n - \gamma}} \, dy
\leq |x-x_0|^{-a} I_\gamma( |\cdot-x_0|^bf)(x). 
\end{multline*}
We can now repeat the above argument for~\eqref{eqn:Ji-estimate}, except that we use the bound for $I_\gamma$ in Theorem~\ref{Theorem 2} to estimate the modular.
This completes the estimate for $J_2$ and so completes the proof.

\section{Poincar\'e and Sobolev inequalities}
\label{section:applications}

In this section, as an application of Theorem \ref{thm:weight}, we show that Poincar\'e and Sobolev inequalities in the variable exponent Morrey spaces are a consequence of the Stein-Weiss inequality. Given a set 
$\Omega$ such that $0 < |\Omega| < \infty$ and a locally integrable function $f$, define
\[
f_\Omega = \frac{1}{|\Omega|} \int_\Omega f(x) \, dx.
\]

We first prove a Poincar\'e  inequality in variable exponent Morrey spaces.

\begin{theorem}\label{thm:poincare}
Let \( \Omega \) be a bounded, open, convex set containing the origin. Let  \( \pp,\, q(\cdot), r(\cdot) \in {\Pp(\Omega)} \) be such that \( \pp, r(\cdot) \in LH_{0}(\Omega) \), \( 1 < p_- \leq p_+ < \infty \), \( 1 < r_- \leq r_+ < \infty \), and \( p(x) \leq q(x)\leq r(x) \) for all \( x \in \Omega \).  Let $\lp : \Omega \to [0,n]$ be such that $0< \lambda_-\leq \lambda_+ <n$ and $\lp \in LH_0(\Omega)$.
Let \(a\) and \(b\) be constants such that conditions \eqref{eqn:condA} and \eqref{eqn:condC} hold for \(\gamma = 1\),  and \eqref{eqn:condB} and \eqref{eqn:condD} with $\gamma=1$ hold for $\pp$ and $r(\cdot)$.
Then, for all $f\in C^1(\Omega)$ there exists a constant $C>0$ such that 
\[
\| |\cdot|^{a} [f - f_\Omega] \|_{M_{q(\cdot)}^{\lp}(\Omega)} \leq C \| |\cdot|^{b} \nabla f \|_{M_{\pp}^{\lp}(\Omega)}.
\]
\end{theorem}

\begin{proof}
Fix a bounded, open, convex set $\Omega$ such that \(0 \in \Omega\). Then, for every $x \in \Omega$ and $f\in C^1(\Omega)$,
\begin{equation}\label{WPI}
|f(x) - f_{\Omega}| \leq I_1(|\nabla f| \chi_{\Omega})(x) 
\end{equation}
holds. (See \cite[Lemma 6.25]{Cruz}.)
Since \( q(x) \leq r(x) \),  we can define the exponent \( s(\cdot) \) by
\[
\frac{1}{q(x)} = \frac{1}{s(x)} + \frac{1}{r(x)}.
\]
By the generalized H\"older's inequality \eqref{Gholder}, Lemma~\ref{lemma:char-estimate},  and  inequality \eqref{WPI}, we get
\begin{align}\label{yar}
\begin{split}
\| | \cdot |^{a} [f - f_{\Omega}] \|_{L^{q(\cdot)}(\Omega)} 
&\leq K \| | \cdot |^{a} [f - f_{\Omega}] \|_{L^{r(\cdot)}(\Omega)} 
\| \chi_{\Omega} \|_{L^{s(\cdot)}(\Omega)}
\\
&\le K (|\Omega| +1 ) \| | \cdot |^{a} [f - f_{\Omega}] \|_{L^{r(\cdot)}(\Omega)}
\\
&\lesssim \| | \cdot |^{a} I_1(|\nabla f| \chi_{\Omega}) \|_{L^{r(\cdot)}(\Omega)} .
\end{split}
\end{align}

Hence, by Lemma~\ref{lemma:equiv-norm} (applied twice), inequality \eqref{yar}, and inequality~\eqref{SWI}, we have that
\begin{align*}
\| |\cdot|^{a} [f - f_\Omega] \|_{M_{q(\cdot)}^{\lambda(\cdot)}(\Omega)} 
&\lesssim  \sup_{x \in \Omega, \, r > 0}  r^{-\frac{\lambda(x)}{q(x)}}  \left\|  |\cdot|^{a} [f - f_\Omega]  \chi_{\tilde{B}(x, r)} \right\|_{L^{q(\cdot)}(\Omega)}
\\
& \lesssim   \sup_{x \in \Omega, \, r > 0}  r^{-\frac{\lambda(x)}{q(x)}}  \| | \cdot |^{a} I_1(|\nabla f| \chi_{\tilde{B}(x, r)}) \|_{L^{r(\cdot)}(\Omega)} 
\\
& \lesssim \| | \cdot |^{a} I_1(|\nabla f| \chi_{\tilde{B}(x, r)}) \|_{M_{r(\cdot)}^{\lambda(\cdot)}(\Omega)}  \\
& \lesssim \| | \cdot |^{b} \nabla f \|_{M_{p(\cdot)}^{\lambda(\cdot)}(\Omega)},
\end{align*}
which is the desired result. 
\end{proof}

We next prove a weighted Hardy-Sobolev inequality in variable exponent Morrey spaces.

\begin{theorem} \label{thm:wHSI-vem}
Let $\Omega \subset \mathbb{R}^n$  be a bounded open set containing the origin. Let \( p(\cdot), q(\cdot) \in \mathcal{P}(\Omega) \) be such that \( p(\cdot), q(\cdot) \in LH_{0}(\Omega) \), \( 1 < p_{-} \le p_{+} < \infty \), \( 1 < q_{-} \le q_{+} < \infty \), and \( p(x) \le q(x) \) for all \( x \in \Omega \). Let $\lp : \Omega \to [0,n]$ be such that $0< \lambda_-\leq \lambda_+ <n$ and $\lp \in LH_0(\Omega)$.
Let \(a\) and \(b\) be constants such that conditions \eqref{eqn:condA}, \eqref{eqn:condB}, \eqref{eqn:condC}, and \eqref{eqn:condD} hold for \(\gamma = 1\). Then, for all $f\in C_c^1(\Omega)$ there exists a constant $C>0$ such that 
\begin{equation}\label{SWI-2}
\| |\cdot|^{a}  f \|_{M_{q(\cdot)}^{\lambda(\cdot)}(\Omega)} 
\leq C \| |\cdot|^{b} \nabla  f \|_{M_{p(\cdot)}^{\lambda(\cdot)}(\Omega)}.
\end{equation}
\end{theorem}

\begin{proof}
For all $f\in C_c^1(\Omega)$, we have that 
\[
|f(x)| \lesssim I_1(|\nabla f|)(x)
\]
holds.  (See~\cite[Lemma~6.26]{Cruz}.)
Therefore, by inequality~\eqref{SWI},
\begin{equation*}
\| |\cdot|^{a} f \|_{M_{q(\cdot)}^{\lambda(\cdot)}(\Omega)} 
\lesssim \left\| |\cdot|^{a} I_1(|\nabla f|) \right\|_{M_{q(\cdot)}^{\lambda(\cdot)}(\Omega)} 
 \le \| |\cdot|^b \nabla f \|_{M_{p(\cdot)}^{\lambda(\cdot)}(\Omega)},
\end{equation*}
which is the desired result. 
\end{proof}

As a corollary to Theorem~\ref{thm:wHSI-vem} we prove a Gagliardo-Nirenberg inequality in variable exponent Morrey spaces.

\begin{theorem} \label{GI}
Let $\Omega \subset \mathbb{R}^n$  be a bounded open set containing the origin. Let \( p(\cdot), p^*(\cdot) \in \mathcal{P}(\Omega) \) be such that \( p(\cdot), p^*(\cdot) \in LH_{0}(\Omega) \), \( 1 < p_{-} \le p_{+} < \infty \), \( 1 < p^*_{-} \le p^*_{+} < \infty \), and \( p(x) \le p^*(x) \) for all \( x \in \Omega \). Let $\lp : \Omega \to [0,n]$ be such that $0 < \lambda_-\leq \lambda_+ <n$ and $\lp \in LH_0(\Omega)$.
Let \(a\) be a constant such that conditions \eqref{eqn:condA}, \eqref{eqn:condB}, \eqref{eqn:condC}, and \eqref{eqn:condD} hold for \(\gamma = 1\), $b=a$, and $q(\cdot)$ replaced by $p^*(\cdot)$.  Fix $\qq\in \Pp(\Omega)$ and $0\leq \theta\leq 1$, and define the exponent $r(\cdot)$ by
\[ \frac{1}{r(x)} = \frac{\theta}{p^*(x)}+\frac{1-\theta}{q(x)}.  \]
Then, for all $f\in C_c^1(\Omega)$ there exists a constant $C>0$ such that for all $f\in C_c^1(\Omega)$, 
    \[
    \| |\cdot|^{a} f \|_{M_{r(\cdot)}^{\lambda(\cdot)}(\Omega)} \leq C \| |\cdot|^{a} \nabla f \|_{M_{p(\cdot)}^{\lambda(\cdot)}(\Omega)}^{\theta} \| |\cdot|^{a} f \|_{M_{q(\cdot)}^{\lambda(\cdot)}(\Omega)}^{1-\theta}.
    \]
\end{theorem}

\begin{proof}
By the generalized H\"older's inequality \eqref{Gholder}, the rescaling property for variable Lebesgue space norms \cite[Proposition 2.18]{Cruz}, and Lemma~\ref{lemma:equiv-norm}, we have that for every $z\in \Omega$ and $r>0$,
\begin{align*}
\| |\cdot|^{a} &f \chi_{B(z,r)}\|_{L_{r(\cdot)}}
\\
&\le K \|\left(|\cdot|^{a}f\right)^{\theta} \chi_{B(z,r)}\|_{L_{\frac{1}{\theta}p^*(\cdot)}} \|\left(|\cdot|^{a}f\right)^{1-\theta} \chi_{B(z,r)}\|_{L_{\frac{1}{(1-\theta)}q(\cdot)}} 
\\
&=  K \| |\cdot|^{a} f \chi_{B(z,r)}\|^{\theta}_{L_{ p^*(\cdot)}} \||\cdot|^{a} f \chi_{B(z,r)}\|^{1-\theta}_{L_{q(\cdot)}}
\\
&=  K r^{\theta \frac{\lambda(z)}{p^{*}(z)}+ (1- \theta) \frac{\lambda(z)}{q(z)}} \left( r^{- \frac{\lambda(z)}{p^{*}(z)}} \| |\cdot|^{a} f \chi_{B(z,r)}\|_{L_{ p^*(\cdot)}}\right)^{\theta} \left( r^{- \frac{\lambda(z)}{q(z)}} \| |\cdot|^{a} f \chi_{B(z,r)}\|_{L_{q(\cdot)}}\right)^{1-\theta}
\\
&\lesssim  r^{ \frac{\lambda(z)}{r(z)}} \| |\cdot|^{a} f \|_{M_{p^*(\cdot)}^{\lambda(\cdot)}\Omega)}^{\theta} \||\cdot|^{a} f \|_{M_{q(\cdot)}^{\lambda(\cdot)}(\Omega)}^{1-\theta}.
\end{align*}
If we now apply Theorem~\ref{thm:wHSI-vem} with $b=a$ and $\qq$ replaced by $p^*(\cdot)$, we get
\begin{equation*}
r^{- \frac{\lambda(z)}{r(z)}}\||\cdot|^{a} f \chi_{B(z,r)}\|_{L_{r(\cdot)}}
\lesssim \| |\cdot|^{a} f \|_{M_{p^*(\cdot)}^{\lambda(\cdot)}(\Omega)}^{\theta} \| |\cdot|^{a}f \|_{M_{q(\cdot)}^{\lambda(\cdot)}(\Omega)}^{1-\theta}
\lesssim  \| |\cdot|^{a} \nabla f \|_{M_{p(\cdot)}^{\lambda(\cdot)}(\Omega)}^{\theta} \| |\cdot|^{a} f \|_{M_{q(\cdot)}^{\lambda(\cdot)}(\Omega)}^{1-\theta}.
\end{equation*}
Since the constant is independent of $z$ and $r$, by Lemma~\ref{lemma:equiv-norm} we get the desired result.
\end{proof}

Finally, we prove a fractional Hardy-Sobolev inequality in variable exponent  Morrey spaces. Before stating the theorem, we recall the definition of the fractional Laplacian. 
For a sufficiently smooth function $f$, the operator $(-\Delta)^s$ can be defined in several equivalent ways. 
One common definition is the integral (or pointwise) form:
\[
(-\Delta)^s f(x) = c \, \mathrm{p.v.} \int_{\mathbb{R}^n} \frac{f(x) - f(y)}{|x-y|^{n+2s}} \, dy,
\]
where p.v. denotes the principal value and $c=c(n,s)$ is a normalization constant.  See~\cite{MR2944369} for further information.

\begin{theorem} \label{thm:fhs}
Let $\Omega \subset \mathbb{R}^n$ be a bounded open set containing the origin. Let \(s \in (0,1) \) and \( p(\cdot), q(\cdot) \in \mathcal{P}(\Omega) \). Assume that \( p(\cdot), q(\cdot) \in LH_{0}(\Omega) \), \( 1 < p_{-} \le p_{+} < \infty \), \( 1 < q_{-} \le q_{+} < \infty \), and \( p(x) \le q(x) \) for all \( x \in \Omega \).  Let $\lp : \Omega \to [0,n]$ be such that $0 < \lambda_-\leq \lambda_+ <n$ and $\lp \in LH_0(\Omega)$.   Let \(a\) and \(b\) be constants such that \eqref{eqn:condA}, \eqref{eqn:condB}, \eqref{eqn:condC}, and \eqref{eqn:condD} hold with $\gamma=2s$.
Then there exists a constant $C>0$ such that for all function $f\in C^2(\Omega)$, 
\begin{equation*}
\| |\cdot|^{a}  f \|_{M_{q(\cdot)}^{\lambda(\cdot)}(\Omega)} 
\leq C \| |\cdot|^{b} \left(-\Delta\right)^{s} f \|_{M_{p(\cdot)}^{\lambda(\cdot)}(\Omega)}.
\end{equation*}
\end{theorem}
\begin{proof}
By the definition of the fractional Laplacian and the properties of the Fourier transform, it follows that for \( s \in [0, 1] \),
\[
I_{2s}\big((-\Delta)^s f\big)(x) = f(x).
\]
The  desired result follows now follows immediately from Theorem~\ref{thm:weight}.
\end{proof}

\begin{remark} 
In Theorem~\ref{thm:fhs}, if fix $0\leq s \leq 1$, $\pp$ such that $p_+<\frac{n-\lambda_+}{2s}$, and let $\qq=\pp$,  $b = 0$ and $a=-2s$, we get a fractional Hardy-Rellich inequality,
\begin{equation*}
\| |\cdot|^{-2s}  f \|_{M_{p(\cdot)}^{\lambda(\cdot)}(\Omega)} \leq C \| \left(-\Delta\right)^{s} f \|_{M_{p(\cdot)}^{\lambda(\cdot)}(\Omega)}.
\end{equation*}
\end{remark}

\bibliographystyle{amsplain}

\bibliography{VEMS}

\end{document}